\documentclass[11 point, a4paper, reqno]{amsart}
\usepackage[english]{babel}
\usepackage{url}
\usepackage{amsthm}
\usepackage{graphicx}
\usepackage{amsfonts}
\usepackage{eufrak}
\usepackage{amsopn}
\usepackage{amsmath, amsfonts, amssymb, amscd, amsthm}
\usepackage{yhmath}
\usepackage{setspace}
\usepackage{mathtools}
\usepackage{multirow}
\usepackage{caption}
\usepackage{graphicx}
\usepackage[all]{xy}
\usepackage{lscape}
\usepackage{comment}
\usepackage{float}
\usepackage{fancyhdr}

\newtheorem{theorem}{Theorem}[section]
\newtheorem{lemma}[theorem]{Lemma}
\newtheorem{proposition}[theorem]{Proposition}

\theoremstyle{definition}
\newtheorem{definition}[theorem]{Definition}
\theoremstyle{remark}
\newtheorem{remark}[theorem]{Remark}
\theoremstyle{example}

\theoremstyle{note}

\numberwithin{equation}{section}
\newcommand{\stirling}[2]{\genfrac{[}{]}{0pt}{1}{#1}{#2}}
\newcommand{\Stirling}[2]{\genfrac{[}{]}{0pt}{0}{#1}{#2}}

\DeclareMathOperator{\GL}{GL}

\DeclareMathOperator{\Span}{Span}

\DeclareMathOperator{\M}{M}

\DeclareMathOperator{\tr}{Tr}

\DeclareMathOperator{\Rank}{Rank}

\begin{document}
\title{On the cardinality of matrices with prescribed rank and partial trace  over a finite field}

\author{Kumar Balasubramanian} 
\thanks{Research of Kumar Balasubramanian is supported by the SERB grant: CRG/2023/000281.}
\address{Kumar Balasubramanian\\
Department of Mathematics\\
IISER Bhopal\\
Bhopal, Madhya Pradesh 462066, India}
\email{bkumar@iiserb.ac.in}

\author{Krishna Kaipa} 
\address{Krishna Kaipa\\
Department of Mathematics\\
IISER Pune\\
Pune, Maharashtra 411008, India}
\email{kaipa@iiserpune.ac.in}

\author{Himanshi Khurana}
\address{Himanshi Khurana\\
Department of Mathematics\\
Harish-Chandra Research Institute\\
Prayagraj, Uttar Pradesh 211019, India}
\email{himanshikhurana@hri.res.in}


\maketitle

\begin{abstract} Let $F$ be the finite field of order $q$ and $\M(n,r, F)$ be the set of $n\times n$ matrices of rank $r$ over the field $F$.  For $\alpha\in F$ and $A\in \M(n,F)$, let $$Z^{\alpha}_{A,r}=\left\{X\in \M(n,r, F)\mid \tr(AX)=\alpha\right \}.$$ 
In this article, we solve the problem of determining the  cardinality of $Z_{A,r}^{\alpha}$. We also solve the generalization of  the problem to rectangular matrices.
\end{abstract}

\section{Introduction}

Let $F$ be the finite field of order $q$. For $n$ a positive integer and $0\leq r\leq n$, we let $\M(n,F)$ to be the vector space  of $n \times n$ matrices with entries in $F$, and $\M(n,r,F)$ to be the subset of $\M(n,F)$ consisting of rank $r$ matrices. For $\alpha \in F$ and $A\in \M(n,F)$, we let $$Z^{\alpha}_{A,r}=\left\{X\in \M(n,r, F)\mid \tr(AX)=\alpha\right \}.$$ 
For $1\leq k\leq n$, we define the \emph{$k$-th partial trace}, or in short the \textit{$k$-trace} of a matrix as the sum of its first $k$ diagonal entries. It is clear that for $X\in \M(n,F)$, the $k$-trace of the matrix $X$ is precisely $\tr(AX)$ where $A$ is the diagonal matrix with the first $k$ diagonal entries being $1$, and the last $(n-k)$ diagonal entries being $0$.  We let $Y_{n,r,k}^{\alpha}$ be the set of $n\times n$ matrices of rank $r$ whose  $k$-trace is $\alpha$. The main result of this paper is a computation of the cardinality of $Y_{n,r,k}^{\alpha}$. It is easy to see that the sets $Z_{A,r}^{\alpha}$ and $Y_{n,r, k}^{\alpha}$, where $k$ is the rank of $A$, have the same cardinality (see Proposition~\ref{cardinality of Z(A,r,alpha) is f(n,r,k,alpha)}). \\

Let $\psi:F \to \mathbb{C}^{\times}$ be a fixed non-trivial additive character of $F$. For $A\in \M(n,F),$ consider 
\begin{equation}\label{gauss sum}\sum_{X\in \M(n,r,F)}\psi(\tr(AX)). 
\end{equation}
In the case  $r=n$, the sum above runs over $\GL(n,F)$ and can be viewed as an analogue for $\GL(n,F)$ of the classical Gauss sums. 
In \cite{Eich}, Eichler computed the above sum \eqref{gauss sum} when $r=n$, and $A=I_{n}$  is the $n\times n$ identity matrix, and used it to compute the cardinality of $Z_{I_{n},n}^{\alpha}$. Eichler's result was reproved by Lamprecht in \cite{Lamp} and more recently by Kim in \cite{Kim} using different methods. See also Theorem 1 in \cite{Ful}  where a similar result is proved. In a more recent work \cite{LiHu}, Li and Hu computed the sum \eqref{gauss sum} for $r=n$ and arbitrary $A\in \M(n,F)$ generalizing the earlier works of Eichler, Lamprecht and Kim. In \cite{Philip}, Buckhiester counted the number of $n \times n $ matrices of rank $r$ and trace $\alpha$ over $F$. This cardinality was also counted independently by Prasad in \cite{Dip[1]} and was used to compute the dimension of a certain representation of $\GL(2n,F)$. In \cite{KumHim[3]}, Balasubramanian, Khurana and Dangodara  computed the cardinality of $n \times n$ matrices of rank $r$ and $1$-trace $\alpha$ over $F$, and used it to calculate the dimension of a certain representation of $\GL(2n,F)$. More recently,  a similar type of calculation was done by  Gorodetsky and Hazan  in \cite{HazGor}, generalizing the work of Prasad, to compute the dimension of a certain representation of $\GL(kn,F)$. Motivated by these works, in this article, we count the number of $n \times n $ matrices of rank $r$ and $k$-trace $\alpha$ over $F$ for $1 \leq k \leq n$. This calculation is motivated by an application to the computation of the dimension formula for a certain twisted Jacquet module of a cuspidal representation of $\GL(2n,F)$. We will discuss the dimension formula in a forthcoming paper. \\

For the sake of clarity, we mention below the statement of the main theorem. Let $f_{n,r,k}^{\alpha}$ denote the cardinality of $Y_{n,r,k}^{\alpha}$.
The cardinality of $\M(n,r,F)$ is denoted 
by $a(n,r,q)$ (see \eqref{eq:Landsberg}).  The Gaussian binomial coefficient ($q$-binomial coefficient) which counts the number of $r$-dimensional linear subspaces of the $n$-dimensional space $F^{n}$ is denoted as $\stirling{n}{r}_{q}$. \\

\begin{theorem}[Main Theorem]\label{main theorem} We have 
\begin{align*}
f_{n,r,k}^1&=  ( a(n,r,q)-g_{n,r,k}  )/q,\\
f_{n,r,k}^0&= ( a(n,r,q) + (q-1) g_{n,r,k})/q,
\end{align*}
where 
\[g_{n,r,k}=
\sum_{i=0}^{r}(-1)^{i} q^{\tbinom{i}{2} + k(r-i)} \,\stirling{k}{i}_{q} \,  a(n-k,r-i,q).\]
\end{theorem}
\vspace{0.2 cm}
The range of summation  in  the expression for $g_{n,r,k}$ above can be taken to be  $\max\{0, r-(n-k)\} \leq i \leq \min\{r,k\}$, as the  terms outside this range are zero. We also remark that 
$g_{n,r,k}=f_{n,r,k}^{0} - f_{n,r,k}^{1}$ equals the sum \eqref{gauss sum} when rank$(A)=k$. The result of the above Theorem can be  succinctly stated in terms of the polynomial generating functions 
$A_n(T), g_{n,k}(T) \in \mathbb{Z}[T]$ defined by  
\[ A_n(T)= \displaystyle \sum_{r=0}^n a(n,r,q)T^r, \qquad g_{n,k}(T)= \displaystyle \sum_{r=0}^n g_{n,r,k} T^r. \]
We have
\[
g_{n,k}(T) = A_{n-k}(q^kT)\, (1-T)(1-qT) \cdots(1-q^{k-1}T).\]
In fact, we first obtain this polynomial identity in Theorem \ref{g_(n,k)}, and use it to obtain the expression for $g_{n,r,k}$ given in Theorem \ref{main theorem}. The polynomial generating functions $f^\alpha_{n,k}(T) = \displaystyle \sum_{r=0}^n f^\alpha_{n,r,k} T^r$ are given by:
\[ f^1_{n,k}(T) = (A_n(T) - g_{n,k}(T))/q, \qquad  f^0_{n,k}(T) = (A_n(T) +(q-1) g_{n,k}(T))/q.\]



\section{Preliminaries}
Throughout, we let $F$ denote  the finite field of order $q$, and we denote by $\M(n,F)$  the set of $n\times n$ matrices with entries in $F$. For $0\leq r\leq n$, we let $\M(n,r,F)$ denote the subset of $\M(n,F)$ consisting of matrices of rank $r$. For $1\leq k\leq n$, and $\alpha \in F$, define
\[Y_{n,r,k}^{\alpha}= \left\{ X=(x_{ij}) \in \M(n,r, F) \mid \sum_{i=1}^{k} x_{ii}=\alpha \right\}.\] 
We write $f_{n,r,k}^{\alpha}$ for the cardinality of $Y_{n,r,k}^{\alpha}$. If $k=n$, to simplify notation, we will denote $Y_{k,r,k}^{\alpha}=Y_{k,r}^{\alpha}$ and $f_{k,r,k}^{\alpha}= f_{k,r}^{\alpha}$.

\begin{proposition}\label{cardinality of Z(A,r,alpha) is f(n,r,k,alpha)} Let $n,r$ be non-negative integers. For $A\in \M(n,F)$ and $\alpha\in F$, let 
\[Z^{\alpha}_{A,r}=\left\{X\in \M(n,r, F)\mid \tr(AX)=\alpha\right \}.\]
If $\Rank(A)=k$, then
\[|Z_{A,r}^{\alpha}|=f_{n,r,k}^{\alpha}. \]
\end{proposition}

\begin{proof} Since $A\in \M(n,k,F)$, there exists $g_{1}, g_{2}\in \GL(n,F)$ such that $g_{1}Ag_{2}^{-1}=B$, where 
\[ B=\begin{bmatrix} I_{k} & 0 \\ 0 & 0\end{bmatrix}.\]
Consider the bijective map 
\[ \phi \colon  \M(n,r,F) \to \M(n,r,F)
,  \quad 
\phi(X)=g_{2}Xg_{1}^{-1}.\]
Since 
\[\tr(B\phi(X))= \tr(Bg_{2}Xg_{1}^{-1})= \tr(g_{1}^{-1}B{g_{2}}X)=\tr(AX),\]
it follows that $\phi$ carries 
$Z_{A,r}^{\alpha}$ bijectively to 
$Y_{n,r,k}^{\alpha}$.
\end{proof}

Before we proceed further, we set up some more notation and recall some more results that we need. We write $a(n,q)$ (respectively $a(n,r,q)$) for the cardinality of $\M(n,F)$ (respectively $\M(n,r,F)$). We let $\stirling{n}{r}_{q}$ denote the Gaussian binomial coefficient and $(T;q)_{n}$ for the $q$-Pochhammer symbol. For the sake of clarity, we record their definitions below.  \\

\begin{definition} Let $n$, $r$ be non-negative integers with $r\leq n$. The $q$-binomial coefficient is defined as  
\[\Stirling{n}{r}_{q}= \frac{[n]_{q}!}{[r]_{q}! [n-r]_{q}!} \]
where $[n]_{q}!= [1]_{q}[2]_{q}\cdots [n]_{q}$ and $[n]_{q}=\frac{q^{n}-1}{q-1}$. 
It is known that $\stirling{n}{r}_{q}$ is a  polynomial of degree $r(n-r)$ in $q$ which evaluates at $q=1$ to 
 the usual binomial coefficient $\tbinom{n}{r}$.
\end{definition}

\begin{definition} Let $n$ be a non-negative integer. We define 
\[(T;q)_{n}= \prod\limits_{i=0}^{n-1}(1-Tq^{i}),\quad  n>0. \]
We take $(T;q)_0=1$.
\end{definition}

We recall the $q$-binomial theorem which goes back to Cauchy. A proof can be found in \cite[\S 1.8]{Stanley1}.
\begin{theorem}[$q$-binomial theorem] \label{qbin_thm} We have 
\[(T;q)_{n}= \sum_{r=0}^{n}\Stirling{n}{r}_{q}(-1)^{r}q^{\tbinom{r}{2}}T^{r}.\]
\end{theorem}
We note that substituting $q=1$ in the $q$-binomial  theorem  recovers the usual binomial theorem.

\begin{proposition} \label{Landsberg}
The number $ a(m \times n, r,q)$ of $m \times n$ matrices of rank $r$ over $F$  equals
\[ a(m \times n, r,q)=
\Stirling{n}{r}_q \Stirling{m}{r}_q  |\GL(r,F)|= \prod\limits_{i=0}^{r-1}\frac{(q^{m}-q^{i}) (q^{n}-q^{i})}{(q^{r}-q^{i})}. \]
 In particular, 
 \begin{equation} \label{eq:Landsberg} a(n,r,q)=a(n \times n, r, q)= \Stirling{n}{r}_q^2 |\GL(r,F)|= \prod\limits_{i=0}^{r-1}\frac{(q^{n}-q^{i})^{2}}{(q^{r}-q^{i})}. \end{equation}
 \end{proposition}
\begin{proof} A  rank $r$ linear map $T:F^n \to F^m$ is uniquely determined by the triple of i) kernel$(T)$, ii) image$(T)$, and iii) the isomorphism $F^n/\text{kernel}(T) \to \text{image}(T)$. These  can be chosen in i) $\stirling{n}{n-r}_q$, ii) $\stirling{m}{r}_q$, and iii) $|\GL(r,F)|$ ways respectively.
\end{proof}

\begin{proposition}\label{cardinality of a(n,r,q) using trace 0 and trace 1} Let $n$ be a positive integer and let $0\leq r\leq n$. Then 
\[f_{n,r,k}^0 + (q-1) f_{n,r,k}^1 = a(n,r,q).\]
\end{proposition}

\begin{proof} It is clear that $ \M(n,r,F)$ can be written as the disjoint union $\displaystyle\bigsqcup_{\alpha\in F}Y_{n,r,k}^{\alpha}$. Thus, we have $\displaystyle a(n,r,q)=\sum_{\alpha\in F}f_{n,r,k}^{\alpha}$. Let $\alpha\neq 0 \in F$. It is enough to show that $f_{n,r,k}^{\alpha}=f_{n,r,k}^{1}$. Let $$Z=\begin{bmatrix}P & Q \\ R & S\end{bmatrix}\in Y_{n,r,k}^{\alpha},$$ where $P\in \M(k,F)$ and so $\tr(P)=\alpha\neq 0$. Clearly the map $\phi\colon Y_{n,r,k}^{\alpha}\rightarrow Y_{n,r,k}^{1}$ defined by $\phi(Z)=KZ$, where $$K=\begin{bmatrix} \alpha^{-1}I_{k} & 0 \\ 0 & I_{n-k}\end{bmatrix}\in \GL(n,F)$$ 
is a bijection. The result follows. 
\end{proof}

For $\alpha\in F$, recall that $f_{k,r}^{\alpha}=f_{k,r,k}^{\alpha}$. The following result is due to Prasad where he computes $f_{k,r}^{0}-f_{k,r}^{1}$. We can use this result in combination with Proposition~\ref{cardinality of a(n,r,q) using trace 0 and trace 1} to explicitly compute the numbers $f_{k,r}^{0}$ and $f_{k,r}^{1}$.

\begin{lemma}[Prasad]\label{DP}
\[ \label{eq:diff1} 
g_{k,r,k}=f_{k,r}^0 - f_{k,r}^1 = (-1)^r q^{\tbinom{r}{2}} \Stirling{k}{r}_q.\]
\end{lemma}

\begin{proof} For a proof, we refer the reader to Lemma 2 in \cite{Dip[1]}.
\end{proof}

We can reformulate Lemma~\ref{DP}, using generating functions and the $q$-binomial theorem as follows.   

\begin{lemma} \label{DP_T}
The polynomial generating function $g_{k,k}(T)= \displaystyle \sum_{r \geq 0} g_{k,r,k} T^r \in \mathbb{Z}[T]$ equals:
\begin{equation} \label{GF for n=k}
g_{k,k}(T)=(1-T)(1-qT)\cdots(1-q^{k-1}T) = (T;q)_{k}.
\end{equation}
\end{lemma}

\section{main theorem}
In this section, we prove the main result of this paper. We continue with the same notation as in the previous section.

\begin{lemma} \label{recursion} $f^\alpha_{n,r,k}$ satisfies the recursion 
\[f^\alpha_{n,r,k}=f^\alpha_{n-1,r,k}q^{2r} + f^\alpha_{n-1,r-1,k}q^{2r-2}(2q^{n-r+1}-1-q) + f^\alpha_{n-1,r-2,k}q^{2r-3}(q^{n-r+1}-1)^{2}\]
for $n>k$.
\end{lemma}

\begin{proof} Let $X \in Y_{n,r,k}^{\alpha}$. Then, $X$ is of the form \[X=\begin{bmatrix} D & v \\ w & x \end{bmatrix}\] where $x \in F$, $v \in F^{n-1}$,  $w$ is a row vector of length $(n-1)$, and $D \in Y^\alpha_{n-1,s,k}$ for some $s \in  \{r,r-1,r-2\}$. We denote the $n \times 1 $ column vector $ \begin{bmatrix} v \\ x \end{bmatrix}$ by $\tilde{v}$. Let $W$ denote the row space of $D$, and let $V$ denote the column space of the 
$n \times (n-1)$ matrix $\begin{bmatrix} D \\ w \end{bmatrix}$.
We consider the cases  $s=r$, $s=r-1$, and $s=r-2$ one by one.
\vspace{0.2 cm}\\
In case $s=r$,  the condition for $X$ to have rank $r$ is that  $w \in W$ and $\tilde{v} \in V$.  Thus, given $D\in Y_{k,r}^{0}$, we get $q^r$ choices each, for the vectors $w$ and $\tilde{v}$ such that $X\in Y_{k+1,r,k}^{0}$. Thus, there are in total $q^{2r}$ choices for the pair $(w, \tilde{v})$ so that $X \in Y^\alpha_{n,r,k}$.
\vspace{0.2 cm}\\
In case $s=r-1$, then  $X$ has rank $r$ if either i) $w \in W$ and $\tilde{v} \notin V$, or ii) $w \notin W$ and $\tilde{v} \in V$. The number of choices for $w$ and $\tilde{v}$ are $q^{r-1}$ and $(q^n-q^{r-1})$ in case i), and 
$(q^{n-1}-q^{r-1})q^r$ in case ii). Thus, there are in total $(2 q^{n+r-1}- q^{2r-2} - q^{2r-1})$ choices for the pair $(w, \tilde{v})$ so that $X \in Y^\alpha_{n,r,k}$.
\vspace{0.2 cm}\\
In case $s=r-2$,   the condition for $X$ to have rank $r$ is that  $w \notin W$ and $\tilde{v} \notin V$. The number of choices for $w$ and $\tilde{v}$ are $(q^{n-1}-q^{r-2})$ and $(q^n-q^{r-1})$ respectively. Thus, there are $(q^{n-1}-q^{r-2})(q^n-q^{r-1})$ choices for the pair $(w, \tilde{v})$ so that $X \in Y^\alpha_{n,r,k}$.
\vspace{0.2 cm}\\
It follows that 
\[ f^\alpha_{n,r,k}=f^\alpha_{n-1,r,k} q^{2r} +f^\alpha_{n-1,r-1,k} q^{2r-2}(2 q^{n-r+1}-1-q) + f^\alpha_{n-1,r-2,k} q^{2r-3}(q^{n-r+1}-1)^2.
\]
\end{proof}

For $1\leq k \leq n$ and $\alpha\in F$, consider the polynomial generating function $f^{\alpha}_{n,k}(T)\in \mathbb{Z}[T]$ defined by 
\begin{equation}\label{eq:definition of f(n,k, alpha)(T)}
f^{\alpha}_{n,k}(T) = \sum_{r\geq 0}f^{\alpha}_{n,r,k} T^{r}.
\end{equation}

\begin{lemma} \label{fnkT_recurrence} For $n>k$, we have 
\[f_{n,k}^\alpha(T)=  (1-T)(1-qT)f_{n-1,k}^\alpha(q^{2}T) + 2q^n T(1-T) f_{n-1,k}^\alpha(qT)  +  q^{2n-1} T^{2}  f_{n-1,k}^\alpha(T) . \]
\end{lemma}

\begin{proof} Using the recursive expression for $f^\alpha_{n,r,k}$ of Lemma  \ref{recursion} in \eqref{eq:definition of f(n,k, alpha)(T)}, we get:
\begin{align}\label{eq:recursion for fn,k}
\notag f_{n,k}^\alpha(T) &= \sum_{r \geq 0} \bigg(f^\alpha_{n-1,r,k} (q^{2}T)^{r} +Tf^\alpha_{n-1,r-1,k} (q^{2} T)^{r-1} (2q^{n-r+1}-1-q) \\
&+ qT^{2}f^\alpha_{n-1,r-2,k} (q^{2}T)^{r-2}(q^{n-r+1}-1)^{2}\bigg). 
\end{align}
In terms  of the index $t=r-1$, the  second term in the above sum can be rewritten as
\begin{multline*}
    2Tq^n\sum_{t \geq 0} (qT)^{t} f^{\alpha}_{n-1,t,k} - (1+q)T \sum_{t \geq 0} (q^{2}T)^{t} f^{\alpha}_{n-1,t,k} =\\
2Tq^{n}f_{n-1,k}^{\alpha}(qT)-(1+q)Tf_{n-1,k}^{\alpha}(q^{2}T).
\end{multline*} 
 Similarly, in terms of the index $s=r-2$,
 the third term in the sum can be rewritten as 
  \begin{multline*}
\sum_{s \geq 0} \left( q^{2n-1} T^{2} f_{n-1,s,k}^{\alpha} T^{s}
+ T^{2} q f_{n-1,s,k}^{\alpha}(q^{2}T)^{s} - 2q^{n}T^{2}f^{\alpha}_{n-1,s,k}(qT)^{s}\right) \\ =
 T^{2}q^{2n-1}f_{n-1,k}^{\alpha}(T)+qT^{2}f_{n-1,k}^{\alpha}(q^{2}T)
-2q^{n}T^{2}f_{n-1,k}^{\alpha}(qT).
 \end{multline*} 
Therefore, \eqref{eq:recursion for fn,k} can be written as 
\[
f_{n,k}^{\alpha}(T) = (1-T)(1-qT)f_{n-1,k}^{\alpha}(q^{2}T) + 2Tq^{n}(1-T)f_{n-1,k}^{\alpha}(qT) +T^{2}q^{2n-1}f_{n-1,k}^{\alpha}(T).\]
\end{proof}

Since  $a(n,r,q) = f^0_{n,r,0}$, it follows that the polynomial generating function 
\[ \displaystyle A_{n}(T) = \sum_{r \geq 0} a(n,r,q)T^{r}\in \mathbb{Z}[T], \]
of the numbers $a(n,r,q)$ (for fixed $n$), satisfies the recurrence:
\begin{equation} 
\label{eq:rec_A}  A_n(T)=  A_{n-1}(q^2T) (1-T)(1-qT) + 2q^nT(1-T)A_{n-1}(qT)  +  q^{2n-1} T^2 A_{n-1}(T), \quad n>0.
\end{equation}

Let $g_{n,r,k}=f_{n,r,k}^{0}-f_{n,r,k}^{1}$. Consider the polynomial generating function:
\[g_{n,k}(T)= \sum_{r \geq 0} g_{n,r,k}T^{r}= f_{n,k}^{0}(T)-f_{n,k}^{1}(T)\in \mathbb{Z}[T],\]
where $f^\alpha_{n,k}(T)$ is as defined in \eqref{eq:definition of f(n,k, alpha)(T)}.
It follows from Lemma \ref{fnkT_recurrence} that 
 $g_{n,k}(T)$ also obeys the same recursion as $f_{n,k}^\alpha(T)$:
 \begin{equation} \label{eq:rec_g}  
g_{n,k}(T) =  (1-T)(1-qT) g_{n-1,k}(q^2T)  + 2q^nT(1-T) g_{n-1,k}(qT)+  q^{2n-1} T^2 g_{n-1,k}(T). 
\end{equation}

For $n=k$, the polynomial $g_{n,k}(T)$ equals $(T;q)_k$ by Lemma \ref{DP_T}. In the next Lemma, we use \eqref{eq:rec_g} to  determine 
$g_{n,k}(T)$ for $n=k+1$ and $n=k+2$.

\begin{lemma}\label{g_k+1} We have
\[\frac{g_{k+1,k}(T)}{(T;q)_{k}} = 1 + (q-1)Tq^{k}= A_1(q^{k}T).\]
and 
\[\frac{g_{k+2,k}(T)}{(T;q)_{k}} = 1 + (q^{2}-1) \left[(Tq^{k})(q+1) +  (Tq^{k})^{2}(q^{2}-q)\right]=A_{2}(q^{k}T). \]
\end{lemma}

\begin{proof} Using the fact that $g_{k,k}(T) =(T;q)_{k}$, Equation \eqref{eq:rec_g} becomes 
\[
g_{k+1,k}(T)=  (1-T)(1-qT)(q^2T;q)_k + 2T(1-T)q^{k+1}(q T;q)_k +T^{2}q^{2k+1}  (T;q)_k. \]
Thus 
\begin{align*}
\frac{g_{k+1,k}(T)}{(T;q)_{k}} &= (1-q^{k}T)(1-q^{k+1}T)+2Tq^{k+1}(1-q^{k}T)+T^{2}q^{2k+1}\\
\vspace{0.2 cm}\\
&= 1+Tq^{k}(q-1) = A_1(q^kT) ,
\end{align*}
because $A_{1}(T)=1+(q-1)T$.
\vspace{0.2 cm}\\
Similarly, $g_{k+2,k}(T)$ equals 
\[  A_{1}(q^{k+2}T)(q^{2}T;q)_{k}(1-T)(1-qT)
+ 2q^{k+2}T(1-T) (qT;q)_{k} A_{1}(q^{k+1}T) + T^{2}q^{2k+3}A_{1}(qT)(T;q)_{k}\]
Therefore, $\tfrac{g_{k+2,k}(T)}{(T;q)_{k}}$ equals
\[
A_{1}(q^{k+2}T)(1-q^{k}T)(1-q^{k+1}T)+2Tq^{k+2}A_{1}(q^{k+1}T)(1-q^{k}T)\\
+q^{3}A_{1}(q^{k}T)(q^{k}T)^{2}. \]
Using $A_{1}(T)=1+(q-1)T$, this simplifies to 
\[\frac{g_{k+2,k}(T)}{(T;q)_{k}} = 1 + (q^{2}-1) \left[(Tq^{k})(q+1) +  (Tq^{k})^{2}(q^{2}-q)\right]=A_{2}(q^{k}T), \]
because $A_2(T)=1+(q+1)^2(q-1) T + (q^2-1)(q^2-q)T^2$.
\end{proof}

This suggests the following theorem. 

\begin{theorem}\label{g_(n,k)}
\[ g_{n,k}(T) = (T;q)_k \; A_{n-k}(q^k T). \]
\end{theorem} 

\begin{proof} We use induction on $n$. The base case is $k=n$ which is true by Lemma~\ref{DP_T}. Assume inductively that the theorem is true for $g_{n-1,k}(T)$. We have
\begin{align*}
g_{n,k}(T)&= (1-T)(1-qT)g_{n-1,k}(q^{2}T) + 2q^{n}T(1-T)g_{n-1,k}(qT) + q^{2n-1}T^{2} g_{n-1,k}(T)\\
&= (1-T)(1-qT)(q^2T;q)_k A_{n-k-1}(q^{k+2}T)  +2q^nT(1-T)(qT;q)_k A_{n-k-1}(q^{k+1} T) \\
&\qquad + q^{2n-1} T^2(T;q)_k \; A_{n-k-1}(q^k T) .\end{align*}
Therefore, 
\begin{align*}
\frac{g_{n,k}(T)}{(T;q)_k} &=
 (1-q^kT)(1-q^{k+1}T)A_{n-k-1}(q^{k+2}T)+ 2Tq^{n}(1-q^{k}T)A_{n-k-1}(q^{k+1}T)\\
& + A_{n-k-1}(q^{k}T)T^{2}q^{2n-1}. 
\end{align*}
In terms of the recurrence relation \eqref{eq:rec_A} for $A_m(Y)$  where $Y=q^kT$ and $m=n-k$, we see that the above expression is just
\[ A_{n-k}(q^kT).\]
This completes the proof of the theorem. 
\end{proof}

\subsubsection*{Proof of Theorem \ref{main theorem}} \hfill  \\
 From Theorem \ref{g_(n,k)} and Theorem \ref{qbin_thm},  it follows that 
\[g_{n,k}(T)=\sum_{i \geq 0} \sum_{j \geq 0} (-1)^{i} q^{\tbinom{i}{2}+kj} \Stirling{k}{i}_{q}a(n-k,j,q)T^{j+i}.\]
We can rewrite this as
$g_{n,k}(T)= \displaystyle \sum_{r \geq 0}  g_{n,r,k} \,T^r$, where:
\[g_{n,r,k} = \sum_{i=0}^r 
(-1)^i q^{\tbinom{i}{2} +k(r-i)} \,\Stirling{k}{i}_q \,  a(n-k,r-i,q).
\]
It follows from Proposition~\ref{cardinality of a(n,r,q) using trace 0 and trace 1} that
\begin{align*}
  f_{n,r,k}^1 &=  ( a(n,r,q)-g_{n,r,k})/q,  \\ 
  f_{n,r,k}^0 &=  ( a(n,r,q)+(q-1) g_{n,r,k})/q.   
  \end{align*}
This completes the proof of Theorem \ref{main theorem}.  \qed
\begin{remark} We can use the results obtained above and recover some earlier known results as particular cases of our result. \\
\begin{enumerate}
\item[a)] When $k=1$,
\begin{small} 
\begin{equation} 
\label{g(n,r,1)} g_{n,r,1}=\sum_{i=0}^{1} (-1)^i q^{\tbinom{i}{2}+(r-i)} \stirling{1}{i}_q a(n-1,r-i,q)=q^ra(n-1,r,q)-q^{r-1}a(n-1,r-1,q).
\end{equation}
\end{small}
Clearly, this is the same as the formula obtained in Lemma 3.5 in \cite{KumHim[3]}. \\

\item[b)] When $k=n$, we have that \[
a(n-k,r-i,q) = \begin{cases}
1 & \text{if}~~ r=i \\
0 & \text{otherwise.}
\end{cases}\]
Thus, we get
\begin{equation} 
\label{g(n,r,n)} g_{n,r,n}=(-1)^r q^{\tbinom{r}{2}} \Stirling{n}{r}_q
\end{equation}
and we recover the formula obtained in Lemma 2 in \cite{Dip[1]}.\\

\item[c)] By Proposition \ref{cardinality of a(n,r,q) using trace 0 and trace 1}, it is easy to see that the difference \[f_{n,r,k}^{0}-f_{n,r,k}^{1}=\sum_{X \in \M(n,r,F)} \psi(\tr(AX))\]
if $\psi$ is a fixed non-trivial additive character of $F$. For $r=n$ and any arbitrary $k$, we obtain that 
\[
g_{n,n,k} = \sum_{i=0}^{k} (-1)^i q^{\tbinom{i}{2}+k(n-i)} \Stirling{k}{i}_q a(n-k,n-i,q).\]
Since $a(n-k,n-i,q)=0$ for $n-i > n-k$, we get that 
\begin{align*}
g_{n,n,k} &= (-1)^k q^{\tbinom{k}{2}+k(n-k)} \Stirling{k}{k}_q a(n-k,n-k,q)\\
&= (-1)^k q^{\tbinom{n}{2}}\prod_{t=1}^{n-k}(q^t-1). 
\end{align*}
Thus we recover the formula obtained in Theorem 2.1 in \cite{LiHu}, and \cite[eq. 1.150]{Stanley1}. 
\end{enumerate}
\end{remark}
\section{The case of rectangular matrices}
The problem of determining the cardinality of 
\[ Z^{\alpha}_{A,r}=  \{X \in \M(n,r,F) \colon \tr(AX)=\alpha \},\]
has a natural generalization to $m \times n$ matrices.  Let
$\M(m \times n,F)$ be the 
 vector space  of $m \times n$ matrices with entries in $F$, and let 
$\M(m \times n,r, F)$ denote the subset of $\M(m \times n,F)$ consisting of matrices of  rank $r$. For a $n \times m$ matrix $A$ over $F$, let $Z^{\alpha}_{A,r}$ again denote (with a little abuse of notation) 
\[ Z^{\alpha}_{A,r}= \{X \in \M(m \times n,r,F) \colon \tr(AX)=\alpha \}.\]
The problem is to determine the cardinality of $Z^{\alpha}_{A,r}$.
 If $k$ denotes rank$(A)$, then just as in Proposition \ref{cardinality of Z(A,r,alpha) is f(n,r,k,alpha)}, the cardinality of $Z^{\alpha}_{A,r}$ is the same as the cardinality of the set 
\[ Y^\alpha_{m\times n,r,k}= \{X \in \M(m\times n,r,F) : X_{11}+\dots+X_{kk}=\alpha \}. \]
Clearly the transpose operation gives a bijection between  $Y^\alpha_{m\times n,r,k}$
and $Y^\alpha_{n\times m,r,k}$, so we assume $m \leq n$ without loss of generality. We recall from Proposition \ref{Landsberg} that \[ |\M(m\times n,r,F)|=a(m \times n,r,q)=\stirling{m}{r}_q \stirling{n}{r}_q |\GL(r,F)|.\] 
We define:
\begin{align*}
    f^\alpha_{m\times n,r,k} &=|Y^\alpha_{m\times n,r,k}|\\
   g_{m\times n,r,k}&=f^0_{m\times n,r,k}- f^1_{m\times n,r,k} \\
   g_{m\times n,k}(T)&=\sum_{r=0}^m 
   g_{m\times n,r,k} T^r  \\
   A_{m \times n}(T)&=\sum_{r=0}^m a(m \times n,r,q) T^r.  
   \end{align*}
   Moreover, as in Proposition \ref{cardinality of a(n,r,q) using trace 0 and trace 1}, we have 
\begin{align} \label{eq:cardinality of a(mxn,r,q) using trace 0 and trace 1} 
 a(m \times n , r, q)&=f^0_{m \times n,r,k} + (q-1) f^1_{m \times n, r, k},\\
  \nonumber f_{m\times n,r,k}^1 &=  ( a(m\times n,r,q)-g_{m\times n,r,k})/q,  \\ 
  \nonumber f_{m \times n,r,k}^0 &=  ( a(m\times n,r,q)+(q-1) g_{m\times n,r,k})/q.   
  \end{align}
Therefore, it suffices to determine the quantities $g_{m \times n, r, k}$, or effectively the polynomial generating function $g_{m \times n, k}(T)$:

\begin{theorem}\label{g_(mxn,k)}
\[ g_{m\times n,k}(T) = (T;q)_k \; A_{(m-k) \times (n-k)}(q^k T). \]
\end{theorem} 
In other words
\begin{equation} \label{eq:gmnrk}
    g_{m \times n,r,k} = \sum_{i=0}^r (-1)^i \stirling{k}{i}_q \,  q^{\tbinom{i}{2} + k(r-i)} \, a( (m-k)\times (n-k),r-i,q).
    \end{equation}
The range of summation in \eqref{eq:gmnrk} can be taken to be $\text{max}\{0, k+r-m\} \leq i \leq \text{min}\{r,k\}$, as the terms outside this range are zero. We will prove Theorem \ref{g_(mxn,k)}  by 
 induction on $m$. The base case is $m=k$ which is treated in the next lemma. It states the  non-obvious fact that $g_{k \times n, r, k}$  is independent of $n$ for all $n \geq k$: 
\begin{lemma} \label{DP_n}
\[g_{k\times n,r,k} = g_{k,r,k} = (-1)^r q^{\binom{r}{2}} \stirling{k}{r}_q, \quad \forall \, n \geq k.\]
\end{lemma}
\begin{proof}
    If $A \in Y^\alpha_{k,i,k}$ 
 i.e., $A_{k \times k}$ 
  has rank $i$ and trace $\alpha$, then the number of matrices  $ X=[ A_{k \times k} \mid B_{k \times n-k} ] $ having  rank $r$ equals $q^{(n-k)i}$ times   $a((k-i) \times (n-k), r-i,q)$, where $a((k-i) \times (n-k), r-i,q) = |\M((k-i) \times (n-k), r-i, F)|$.
Therefore,  
  \[f^\alpha_{k\times n,r,k}
=\sum_i f^\alpha_{k,i,k} \cdot  
 q^{(n-k)i} \,  a((k-i) \times (n-k), r-i,q).\]
It follows that 
  \begin{multline*}
g_{k\times n,r,k}
=\sum_i g_{k,i,k} \cdot  
 q^{(n-k)i} \,  a((k-i) \times (n-k), r-i,q) \\ =  \sum_i (-1)^i q^{\binom{i}{2}} \cdot  
 q^{(n-k)i}  \stirling{k}{i}_q \stirling{k-i}{r-i}_q \stirling{n-k}{r-i}_q |\GL(r-i,F)|,\end{multline*}
where we have used Lemma \ref{DP}. Using the identity $\stirling{k}{i}_q \stirling{k-i}{k-r}_q=\stirling{k}{r}_q \stirling{r}{i}_q$, we see that 
\[ \tilde{g}_{k\times n,r,k} :=  \frac{g_{k\times n,r,k}}{\stirling{k}{r}_q}=
\sum_i (-1)^i q^{\binom{i}{2}}  \stirling{r}{i}_q  \stirling{n-k}{r-i}_q |\GL(r-i,F)|  q^{(n-k)i}, 
\]
for a fixed $r$, only depends on $n-k$ for all $k, n$ with $n \geq k \geq r$. We will prove the assertion in the statement  by induction on $n$. The base case $n=k$ is true for all $k$. We inductively assume that the assertion holds for $n-1$. Since $\tilde{g}_{k\times n, r, k} = \tilde{g}_{(k-1) \times (n-1), r, (k-1)}$ which by the inductive hypothesis equals 
  \[ \frac{g_{(k-1)\times (n-1),r,k-1}}{\stirling{k-1}{r}_q}  = \frac{(-1)^r q^{\binom{r}{2}} \stirling{k-1}{r}_q}{\stirling{k-1}{r}_q}=(-1)^r q^{\binom{r}{2}}, \]
  we have shown that 
   \[ \frac{g_{k\times n,r,k}}{\stirling{k}{r}_q}=(-1)^r q^{\binom{r}{2}},\]
  as required.
  \end{proof}
 \begin{proof} (of Theorem \ref{g_(mxn,k)}) \hfill \\
 Having proved the base case $m=k$, of the result, we  
 assume inductively that the theorem is true for $g_{m\times n,k}(T)$. 
  Let 
$X = \begin{bmatrix}
        B_{m \times n}\\ w_{1 \times n} 
    \end{bmatrix}$ be an element of  $Y^\alpha_{(m+1) \times n,r,k}$. Then, either 
$B \in Y^\alpha_{m \times n,r,k}$ with $w$ in the row space of $B$, or $B \in Y^\alpha_{m \times n,r-1,k}$ with $w$ not in the row space of $B$. Thus
\[ f^\alpha_{(m+1) \times n, r, k}= q^r f^\alpha_{m \times n, r, k} +(q^n-q^{r-1}) f^\alpha_{m \times n, r-1, k}. \]
In terms of the generating function $f^\alpha_{(m+1) \times n,  k}(T)$, we have the recursive relation
\[ f^\alpha_{(m+1) \times n,  k}(T)=
f^\alpha_{m \times n,  k}(qT) +q^n T f^\alpha_{m \times n,  k}(T) - T f^\alpha_{m \times n,  k}(qT). \]
We may rewrite this as
\begin{equation} \label{eq:recursion for fmxn,k}
    f^\alpha_{(m+1) \times n,  k}(T)=
(1-T) f^\alpha_{m \times n,  k}(qT) +q^n T f^\alpha_{m \times n,  k}(T) , \quad m \geq k. \end{equation}
Since  $g_{m\times n,r,k}=f^0_{m\times n,r,k}- f^1_{m\times n,r,k}$ we have
\begin{equation} \label{eq:recursion for g(mxn,k)}
g_{(m+1) \times n,  k}(T)=
(1-T) g_{m \times n,  k}(qT) +q^n T g_{m \times n,  k}(T). 
\end{equation}
We make another observation from 
\eqref{eq:recursion for fmxn,k}:
Since $a(m\times n, r,q) = f^0_{m\times n, r, 0}$, it follows that $A_{(m+1) \times n}(T) = f^0_{(m+1)\times n, 0}(T)$ satisfies the recursion:
\begin{equation} \label{eq:recursion for a(mxn,k)}
    A_{(m+1) \times n}(T)=
(1-T) A_{m \times n}(qT) +q^n T A_{m \times n}(T) , \quad m \geq 0. \end{equation}

Using the inductive hypothesis in \eqref{eq:recursion for g(mxn,k)}, we get:
\[
g_{(m+1) \times n,  k}(T)=
(1-T) (qT;q)_k A_{(m-k) \times (n-k)}(q^{k+1}T)
+q^n T (T;q)_k A_{(m-k) \times (n-k)}(q^k T).
\]
We may rewrite this as:
\[
\frac{g_{(m+1) \times n,  k}(T)}{(T;q)_k}=
 (1-q^kT) A_{(m-k)\times (n-k)}(q (q^{k}T))
+ q^{n-k} (q^kT)   A_{(m-k)\times(n-k)}(q^k T).\]

Using \eqref{eq:recursion for a(mxn,k)} in the right hand side of the above equation, we get: 

\[ \frac{g_{(m+1) \times n,  k}(T)}{(T;q)_k}=
A_{(m+1-k)\times (n-k)}(q^k T)
\]
This completes the proof of the theorem. 
\end{proof}
\subsection{A different proof of Theorem \ref{g_(mxn,k)}}
The induction in the proof of Theorem \ref{g_(mxn,k)} was on the parameter $m$ with $k$ being fixed. We now develop a recurrence for $g_{m \times n, r, k}$ which does not fix any of the parameters $m,n,r,k$, and use it to obtain a  short proof of Theorem \ref{g_(mxn,k)}. 

\begin{lemma} The quantities $f^{\alpha}_{m \times n, r, k}$ and $g_{m \times n, r, k}$ obey the recurrence:
\begin{align} \label{eq:f_new_recur}
  f^\alpha_{(m+1) \times (n+1), r+1, k+1}&= f^\alpha_{m\times n, r+1, k} \,q^{r+1} - f^\alpha_{m\times n, r, k} \,q^r\\  
  \nonumber &\quad +(a(m\times (n+1), r+1,q)-a(m \times n, r+1, q))q^{r}\\ 
  \nonumber &\quad +  a(m\times (n+1), r,q)(q^{n} -q^{r-1}) + a(m\times n, r,q)\,q^{r-1}
  .\\
 \label{eq:g_new_recur} g_{(m+1) \times (n+1), r+1, k+1}&=
   g_{m\times n, r+1, k}\, q^{r+1} - g_{m\times n, r, k} \, q^r.     
  \end{align}
    \end{lemma}
    \begin{proof}
        The recurrence \eqref{eq:g_new_recur} 
         immediately follows from the recurrence \eqref{eq:f_new_recur}. So it suffices to prove \eqref{eq:f_new_recur}.
         Given $X \in \M((m+1) \times (n+1), r+1, F)$ of the form 
         \[X= \begin{pmatrix}  w\\U
            \end{pmatrix},
         \]
         with 
        $w \in \M(1 \times (n+1), F)$, and  
         $U \in \M( m \times (n+1),F)$. We note that either 1) $U$ has rank $r+1$ or 2) $U$ has rank $r$. In each of these two cases we consider the subcases a) the first column of $U$ is not the zero vector, b) the first column of $U$ is  the zero vector. Let $\tau_U = U_{12}+ U_{23}+ \dots +U_{k,k+1}$. We note that the number of matrices $U$ of rank $s$ with $\tau_U=\beta$ is $f^\beta_{m \times (n+1), s, k}$ because $\tau_U=\tr(AU)$ for some  $(n+1) \times m$
matrix $A$ of rank $k$.\\
         
         In case 1a), the condition for $X$ to be in $Y^\alpha_{(m+1)\times(n+1), r+1, k+1}$ is that  
         $w$ must be in the row space of $U$ with $w_1 = \alpha - \tau_U$. Thus, there are $q^r$ choices for $w$, for each of the  $a(m \times (n+1), r+1,q)-a(m \times n, r+1,q))$ choices for $U$.\\
         
         In case 1b), the condition for $X$ to be in $Y^\alpha_{(m+1)\times(n+1), r+1, k+1}$ is that  
         $w$ must be in the row space of $U$, and $\tau_U$ must be $\alpha$. Thus, the number of choices for $w$ is $q^{r+1}$, for each of the $f^\alpha_{m \times n, r+1,k}$ choices of $U$.\\
         
         In case 2a), the condition for $X$ to be in $Y^\alpha_{(m+1)\times(n+1), r+1, k+1}$ is that  
         $w$ is a vector  not in the row space of $U$, and  $w_1 = \alpha - \tau_U$. The number of choices for $w$ is $(q^{n}-q^{r-1})$, for each of the $(a(m \times (n+1), r,q) - a(m \times n,r,q))$ choices for $U$.\\

         In case 2b), the condition for $X$ to be in $Y^\alpha_{(m+1)\times(n+1), r+1, k+1}$ is that  
         $w$  must not be in the row space of $U$
         and $w_1=\alpha-\tau_U$. 
 If $\tau_U=\alpha$, then there are $(q^n-q^r)$ choices for $w$ for each of the $f^\alpha_{m \times n, r, k}$ choices of $U$. 
 If $\tau_U \neq \alpha$, then there are $q^n$ choices for $w$, for each of the $a(m \times n, r,q)-f^\alpha_{m \times n, r, k}$ choices of $U$. \\
 
 Thus $f^\alpha_{ (m+1)\times (n+1),r+1,k+1}$ equals
 \begin{multline*}
     q^r(a(m \times (n+1), r+1,q)-a(m \times n, r+1,q))
     + q^{r+1} f^\alpha_{m \times n, r+1,k}\\
     + (q^{n}-q^{r-1}) 
     (a(m \times (n+1), r,q) - a(m \times n,r,q)) \\
     +(q^n-q^r) f^\alpha_{m \times n, r, k}
     +q^n (a(m \times n, r,q)-f^\alpha_{m \times n, r, k}),    
 \end{multline*} 
which is the same as \eqref{eq:f_new_recur}.          \end{proof}
Using \eqref{eq:g_new_recur} in the generating function
\[g_{(m+1)\times(n+1),k+1}(T)=1 + \sum_{r \geq 0} g_{(m+1)\times(n+1),r+1,k+1} T^{r+1}, \]
we get 
\begin{align*}
g_{(m+1)\times(n+1),k+1}(T)&=1 + \sum_{r \geq 0} T^{r+1} (g_{m\times n, r+1, k} q^{r+1} - g_{m\times n, r, k} q^r) \\
&= g_{m\times n,k}(qT)-Tg_{m\times n,k}(qT)=(1-T)g_{m\times n,k}(qT)
\end{align*}
The recurrence $g_{(m+1)\times(n+1),k+1}(T)=(1-T)g_{m\times n,k}(qT)$ repeated $k$ more times  gives
\begin{align*} g_{(m+1)\times(n+1),k+1}(T)&=(1-T)(1-qT)\cdots (1-q^kT) g_{(m-k)\times(n-k),0}(q^{k+1}T) \\ &=(T;q)_{k+1} A_{(m-k)\times(n-k)}(q^{k+1}T), 
\end{align*}
which is the same as
\[ g_{m\times n,k}(T) = (T;q)_k \; A_{(m-k) \times (n-k)}(q^k T).  \qed \]

\subsection{Remarks on literature:}
While revising this paper, we came to know that the problem of determining $f^\alpha_{m \times n, r, k}$ and $g_{m \times n, r, k}$ has been studied in the coding theory community, for example by Delsarte \cite{Delsarte1}, Ravagnani \cite{Ravagnani}, and Beelen-Ghorpade  \cite{Beelen_Ghorpade}. The approach of Delsarte realizes $g_{m \times n, r ,k}$  
 as the eigenvalues of a certain   \emph{association scheme} on $\M(m \times n, F)$ with rank as metric.
The formula for $g_{m \times n, r ,k}$  
obtained in \cite[Theorem A.2]{Delsarte1}  is 
\begin{equation} \label{eq:gmnrk_Delsarte}
g_{m \times n, r ,k}=\sum_{i=0}^m (-1)^{r-i} q^{in + \tbinom{r-i}{2}} \stirling{m-i}{m-r}_q    \stirling{m-k}{i}_q
\end{equation}
In \cite[Theorem 65]{Ravagnani}, Ravagnani gives another proof of \eqref{eq:gmnrk_Delsarte}, as an application of the  MacWilliams identities for Delsarte rank-metric codes.
More recently, Beelen and Ghorpade, while studying codes associated to determinantal varieties, obtain the following  formula (\cite[Proof of Theorem A.7]{Beelen_Ghorpade}): 
\begin{equation} \label{eq:gmnrk_Beelen_Ghorpade}
g_{m \times n, r ,k}=\sum_{i=0}^k 
(-1)^i q^{\tbinom{i}{2}+i(m-k)} \stirling{k}{i}_q \stirling{m-k}{r-i}_q  
\prod_{t=0}^{r-i-1} (q^{n-i}-q^t),  
\end{equation}
Their approach is elementary,  uses only linear algebra and induction. They also show using results of \cite{Delsarte2} that \eqref{eq:gmnrk_Beelen_Ghorpade} is 
equivalent to \eqref{eq:gmnrk_Delsarte}. \\

We now show that our formula 
\eqref{eq:gmnrk} is equivalent to 
\eqref{eq:gmnrk_Delsarte}. In \cite[p.267]{Delsarte2}, it is shown that  the following two expressions represent the same function (known as a $q$-Kravchuk polynomial $F_c(k,r,m) \in \mathbb{R}[q]$ with a real parameter $c>q^{-1}$:
\begin{align*}
F_c(k,r,m)&=\sum_{i=0}^r 
(-1)^{r-i}  (cq^m)^i  q^{\tbinom{r-i}{2}} \stirling{m-i}{m-r}_q \stirling{m-k}{i}_q \\
&=\sum_{i=0}^r 
(-1)^{i}  q^{\tbinom{i}{2}} \stirling{k}{i}_q \stirling{m-k}{r-i}_q  \prod_{t=0}^{r-i-1} (cq^m-q^{k+t}). 
\end{align*}
For $c=q^{n-m}$, the first expression yields 
\eqref{eq:gmnrk_Delsarte}, whereas the second expression yields \eqref{eq:gmnrk}.
 
 \section{Some properties of $g_{m \times n,r,k}$}  
An interesting question is to determine how the sizes
$|Z^{\alpha}_{A,r}|$ depend on $\alpha$. For example, if $A \neq 0$ then
the set $\{X \in  \M(m\times n, F) \colon \tr(AX)=0\}$ is a codimension one linear subspace of the vector space $\M(m \times n , F)$, and hence has cardinality $q^{mn-1}$. This cardinality is also  the sum 
\[ \sum_{r=0}^m |Z^{\alpha}_{A,r}| = \sum_{r=0}^m f^\alpha_{m \times n, r, k},\] and hence $\displaystyle \sum_{r=0}^m(f^0_{m \times n, r, k}-f^1_{m \times n, r, k}) = \sum_{r=0}^m g_{m \times n, r, k}$ must be zero. This can be seen as follows:  $\displaystyle\sum_{r=0}^m g_{m \times n, r, k}$ is the evaluation of polynomial
\[ g_{m \times n, k}(T) = (T;q)_k A_{(m-k) \times (n-k)}(T), \] 
at $T=1$, and $(T;q)_k=(1-T)(1-qT) \cdots (1-q^{k-1}T)$  evaluates to $0$ at $T=1$ for $k>0$.\\

We recall that 
$f^\alpha_{m \times n, r, k}=f^1_{m \times n, r, k}$  for all $\alpha \in F\setminus \{0\}$. Thus, the quantity 
\begin{equation}
    h_{m\times n,r,k} = \frac{g_{m\times n ,r,k}}{a(m\times n, r, q)},
\end{equation}
is a measure of equidistribution of the function $\alpha \mapsto |Z^{\alpha}_{A,r}|$ where rank$(A)=k$. Some natural questions arise here: \begin{enumerate}
    \item How does the sign of $h_{m \times n, r, k}$ depend on the parameters $m,n,r,k$?
    \item Can we find an upper  bound on  $|h_{m \times n, r, k}|$?\\
    For fixed $m,n,r,k$ what is the asymptotic behaviour of  $|h_{m \times n, r, k}|$ as $q \to \infty$?
\end{enumerate}
 We will answer these questions.  A non-obvious property of $h_{m \times n, r, k}$ is that it is symmetric in $r$ and $k$: 
\begin{lemma} [Delsarte \cite{Delsarte1, Delsarte2}]
\label{h_symmetry} \[
    h_{m\times n,r,k} = h_{m\times n,k,r}.\]
\end{lemma}
\begin{remark} Before going to our proof of this result, we quickly explain Delsarte's  proof.  
Delsarte showed in \cite[p.240]{Delsarte1} and \cite[eqn. (16)]{Delsarte2}, that $h_{m\times n,r,k}$ can be expressed in the terms of the \emph{basic hypergeometric function} ${}_2\phi_2$ (the definition of ${}_2\phi_2$ is as in references [1], [13] of \cite{Delsarte2}):
\begin{equation}
    h_{m\times n,r,k} = {}_2\phi_2( (\begin{smallmatrix}
        q^{-r}, & q^{-k} \\ q^{-m}, & q^{-n}
    \end{smallmatrix}); q;q)=\sum_{j \geq 0} \frac{(q^{-r};q)_j (q^{-k};q)_j q^j}{(q^{-m};q)_j (q^{-n};q)_j (q;q)_j}
\end{equation}
It is clear that the above expression is symmetric in $r$ and $k$. We can rewrite this expression as the 
 alternating sum:
\[ 
    h_{m\times n,r,k} =\sum_{j \geq 0} \frac{ q^{j(m+n+1-r-k)} \,  a(r \times k, j,q)}{(q;q)_j \,  a(m \times n, j,q)}. 
\]
\end{remark}
We now return to our proof of Lemma \ref{h_symmetry}.
\begin{proof} 
Dividing the expression  \eqref{eq:gmnrk} for  $g_{m\times n,r,k}$ by $a(m \times n, r,q)$ we get 
\[h_{m\times n,r,k}=
\sum_{i=0}^{\text{min}\{r,k\}} (-1)^i   q^{\tbinom{i}{2} + k(r-i)}
\frac{\stirling{k}{i}_q \stirling{m-k}{r-i}_q \stirling{n-k}{r-i}_q |\GL(r-i,F)|}{\stirling{m}{r}_q \stirling{n}{r}_q |\GL(r,F)|}
 \]
The ratio in the right hand side can be simplified as 
\[\frac{ 
\dfrac{(q^{k-i+1};q)_i}{(q;q)_i} \cdot 
\dfrac{(q^{m+i+1-r-k};q)_{r+k-i}}{(q^{m-k+1};q)_{k} (q;q)_{r-i}} \cdot 
\dfrac{(q^{n+i+1-r-k};q)_{r+k-i}}{(q^{n-k+1};q)_{k} (q;q)_{r-i}} \cdot 
q^{\binom{r-i}{2}} (q;q)_{r-i} (-1)^{r-i}}
{ \dfrac{(q^{m-r+1};q)_{r}}{(q:q)_r}  \cdot 
\dfrac{(q^{n-r+1};q)_{r}}{(q:q)_r} \cdot 
q^{\binom{r}{2}} (q;q)_{r} (-1)^{r}}.
\]
This can be rewritten as 
\[ \frac
{ 
(q^{k-i+1};q)_i (q^{r-i+1};q)_i
(q^{m+i+1-r-k};q)_{r+k-i}
(q^{n+i+1-r-k};q)_{r+k-i}
}
{
(-1)^i q^{\binom{i}{2} + i(r-i)} (q;q)_i (q^{n-r+1};q)_{r} (q^{m-r+1};q)_{r} (q^{m-k+1};q)_{k}
(q^{n-k+1};q)_{k}
}\]
Therefore $h_{m\times n,r,k}$ equals
\begin{equation} \label{eq:hmnrk_symmetric}
 \displaystyle\sum_{i=\text{max}\{0,r+k-m\}}^{\text{min}\{r,k\}}  \!\!\!\!\frac{ \dfrac{q^{(k-i)(r-i)}}{(q;q)_i}
(q^{k-i+1};q)_i (q^{r-i+1};q)_i
(q^{m+1+i-r-k};q)_{r+k-i}
(q^{n+1+i-r-k};q)_{r+k-i}
}
{ (q^{m-r+1};q)_{r} (q^{m-k+1};q)_{k}
(q^{n-r+1};q)_{r} (q^{n-k+1};q)_{k}
},
\end{equation}
which is clearly symmetric in $r$ and $k$.
\end{proof}

\begin{lemma} \label{alternating_sum}  For $q>2$, the terms in the alternating sum 
\eqref{eq:gmnrk}
\[     g_{m \times n,r,k}=\sum_{i=\mathrm{max}\{0, k+r-m\}}^{\mathrm{min}\{r,k\}}
\!\!\!\!(-1)^i 
 \stirling{k}{i}_q \,  q^{\tbinom{i}{2} + k(r-i)} \, a( (m-k)\times (n-k),r-i,q),\]
are decreasing in absolute value.
\end{lemma}
\begin{proof}
The absolute value of the ratio of the $(i+1)$-th term to the $i$-th term in the sum above simplifies to 
\begin{equation}
    \label{eq:ratio}
 \frac{
(1-\dfrac{1}{q^{k-i}})(1-\dfrac{1}{q^{r-i}})
}
{ \dfrac{(q^{i+1}-1)(q^{m-k-r+i+1}-1)}{q}  (q^{n-k-r+i+1}-1)},    \end{equation}
for  $i$ in the range 
\[ \text{max}\{0, k+r-m\} \leq i \leq \text{min}\{r,k\}-1. \]
In this range the numerator of \eqref{eq:ratio} is clearly less than $1$, and hence it suffices to show that the denominator is $\geq 1$. There are two possible cases:  $k+r > m$ or  $k+r \leq m$.\\

In case $k+r >m$,  then $i \geq k+r-m \geq 1$, and hence  
\[ (q^{n-k-r+i+1}-1)  \geq (q^{m-k-r+i+1}-1)\geq (q-1) \geq 1,\]
and 
\[(q^{i+1}-1)/q \geq (q^2-1)/q >1,\]
which shows that the denominator is $ \geq 1$.\\

If $k+r \leq  m$, then 
\[ (q^{n-k-r+i+1}-1)  \geq (q^{m-k-r+i+1}-1)\geq (q^{i+1}-1) \geq 1,\]
and 
\[\frac{(q^{i+1}-1)(q^{m-k-r+i+1}-1)}{q} \geq \frac{(q^{i+1}-1)^2}{q} \geq  \frac{(q-1)^2}{q} > 1,\]
which shows that the denominator is $\geq 1$, when $q>2$. If $q=2$, the denominator is $\geq 1$ unless $n=m=r+k$ and $i=0$.\\

Suppose $q=2$ and $m=n=r+k$. Writing \eqref{eq:gmnrk} as 
\[g_{m \times n, r, k} = e_0-e_1+e_2- \dots +(-1)^r e_r,\]
we have shown $\frac{e_{i+1}}{e_i} < 1$ for $i\geq 1 $. As for $\frac{e_1}{e_0}$, we note that 
\[\frac{e_1}{e_0} =
2
(1-\frac{1}{2^{k}})(1-\frac{1}{2^{r}}),\]
is $<1$ only if $(k-1)(r-1)=0$. In other words, $\frac{e_1}{e_0}> 1$ if and only if $m=n=r+k$ and $r,k \geq 2$. In this case 
\[  e_0-e_1 < g_{m \times n, r, k} < e_0-e_1 +e_2.\]
We have
\begin{multline*}
    \frac{e_0-e_1 +e_2}{|\GL(r,F)| \, 2^{kr}}=1-\frac{(2^k-1)(2^r-1)}{2^{m-1}}+\frac{(2^k-1)(2^r-1)(2^{k-1}-1)(2^{r-1}-1)}{  2^{2m-4}\, 27 } \\ < 1- \frac{25(2^k-1)(2^r-1)}{2^{m-1} 27}, \end{multline*}
where we have used the fact that $(2^{k-1}-1)(2^{r-1}-1)<2^{k+r-2}=2^{m-2}$. Since $k+r=m$ and $k,r \geq 2$, the minimum value of $(2^k-1)(2^r-1)$ is $3(2^{m-2}-1)$, and hence
\[1- \frac{25(2^k-1)(2^r-1)}{2^{m-1} 27}< \frac{-7}{18} 
 +\frac{25 }{18 \cdot 2^{m-2}} 
 < \frac{-1}{24},\]
where we have used the fact that $m=r+k \geq 4$. Thus $e_0-e_1+e_2<0$, which shows that $g_{m \times n , r, k}<0$ in this case.
\end{proof}

\begin{proposition}  For $q>2$, we have:\label{h_sign}
     \begin{equation} \label{eq:h_sign}
     \mathrm{sgn}(h_{m\times n,r,k})= (-1)^{\mathrm{max}\{r+k-m,0\}}
     \end{equation}
     (If we do not assume $m \leq n$, then $m$ should be replaced by   $\mathrm{min}\{m,n\}$ in this assertion.)\\
     If $q=2$, the result still holds unless $m=n=r+k$ and $r,k \geq 2$, in which case $\mathrm{sgn}(h_{m\times n,r,k})= -1$.
     \end{proposition}
\begin{proof}
Since $h_{m \times n,r,k}=g_{m\times n,r,k}/a(m\times n,r,q)$, it suffices to show that $g_{m\times n,r,k}$ has the asserted sign. If $q>2$, then by  Lemma \ref{alternating_sum}, the terms of the alternating sum \eqref{eq:gmnrk} are decreasing in absolute value, and hence the sign of $g_{m \times n,r,k}$ is the sign of the first (non-zero) term in the alternating sum. This is the term with index $i=\ell$ where $\ell=\text{max}\{0, k+r-m\}$. Thus, sgn$(g_{m\times n,r,k})= (-1)^{\ell}$.\\

If $q=2$, we have shown in the last part of the proof of Lemma \ref{alternating_sum}, that \eqref{eq:h_sign}  is violated if and only if $m=n=r+k$ and $r,k \geq 2$.
\end{proof}
In fact, Lemma \ref{alternating_sum} can be used to give an upper bound on  the absolute value of $h_{m \times n , r,k}$:
\begin{theorem}
Let $q>2$, and let $\ell=\mathrm{max}\{0,k+r-m\}$.
\begin{equation} \label{eq:h_bound}
    |h_{m \times n , r,k}| \leq q^{-kr -\ell(n-m) + \binom{\ell}{2}} \, 
\frac{(\frac{1}{q^k};q)_{\ell} \, (\frac{1}{q^r};q)_{\ell} \, (\frac{1}{q^m};q)_{r+k-\ell} \,(\frac{1}{q^n};q)_{k+r-\ell}}{(\frac{1}{q^\ell};q)_{\ell} \,(\frac{1}{q^m};q)_{k} \, (\frac{1}{q^m};q)_{r} \, (\frac{1}{q^n};q)_{r} \,(\frac{1}{q^n};q)_{k}}. 
\end{equation}
In particular,
\begin{equation} \label{eq:h_asymptotic_bound}
  |h_{m \times n , r,k}|= O( q^{-kr -\ell(n-m) +\binom{\ell}{2}} )  \quad \text{ as $q \to \infty$}.  
\end{equation}
\end{theorem}
\begin{proof}
   We note that all the Pochhammer symbols in the right hand side of  \eqref{eq:h_bound} lie in the interval $(0,1]$ and tend to $1$ as $q \to \infty$:  this is because for natural numbers  $a<b$ we have 
   \[(\tfrac{1}{q^b};q)_a = (1-\tfrac{1}{q^b})(1-\tfrac{1}{q^{b-1}}) \cdots (1-\tfrac{1}{q^{b-a+1}}).\]
Therefore, the assertion \eqref{eq:h_asymptotic_bound} readily follows from    \eqref{eq:h_bound}. As for the assertion in \eqref{eq:h_bound}, we again use the fact that the expression \eqref{eq:hmnrk_symmetric} is an alternating sum  of decreasing terms (by Lemma \ref{alternating_sum}),  and hence $|h_{m \times n,r,k}|$  is bounded above by the absolute value of the first (nonzero) term of the sum, which is the term with index $i=\ell$.
Using the identity $(q^a;q)_b= (\frac{1}{q^{a+b-1}};q)_b (-1)^b q^{ \binom{b}{2}+ab}$ 
 for natural numbers $a,b$, we can express the absolute value of the $i$-th term of  \eqref{eq:hmnrk_symmetric} as
\[ (\tfrac{1}{q})^{\left(rk -\binom{i}{2} +i(i-(r+k-m)) + i(i-(r+k-n)) \right)} \,  \frac{(\frac{1}{q^k};q)_{i} \, (\frac{1}{q^r};q)_{i} \, (\frac{1}{q^m};q)_{r+k-i} \,(\frac{1}{q^n};q)_{k+r-i}}{(\frac{1}{q^i};q)_{i} \, (\frac{1}{q^m};q)_{k} \, (\frac{1}{q^m};q)_{r} \, (\frac{1}{q^n};q)_{r} \,(\frac{1}{q^n};q)_{k}}. 
\]
Taking $i=\ell$, we obtain \eqref{eq:h_bound}.

\end{proof}
We remark that the exponent 
\begin{equation}
E=kr +\ell(n-m) -\binom{\ell}{2}    
\end{equation}
 of $\tfrac{1}{q}$ in \eqref{eq:h_asymptotic_bound} is positive: if $\ell=0$, we have 
 $E=kr >0$. If $\ell=k+r-m$, then we  proceed as follows: by symmetry of $E$ in $r, k$, we may assume $r \leq k$ without loss of generality. We can express  
\[E=\frac{r(2k-r+1)}{2} + \binom{m-k}{2} +\ell(m-k) + \ell(n-m),\]
which is non-negative because $n \geq m \geq k \geq r \geq 0$. 
\bibliographystyle{amsplain}
\bibliography{cardinality}

\end{document}